\documentclass[12pt]{article}
\usepackage{amsmath,amsfonts,amssymb,amsthm}
\usepackage{epsfig}

\theoremstyle{plain}
\newtheorem{theorem}{Theorem}[section]
\newtheorem{corollary}[theorem]{Corollary}
\newtheorem{lemma}[theorem]{Lemma}
\newtheorem{proposition}[theorem]{Proposition}

\theoremstyle{definition}
\newtheorem{definition}[theorem]{Definition}

\numberwithin{equation}{section}

\newcommand{\R}{{\mathbb R}}
\newcommand{\N}{{\mathbb N}}

\providecommand{\vint}[1]{\mathchoice
          {\mathop{\vrule width 5pt height 3 pt depth -2.5pt
                  \kern -9pt \kern 1pt\intop}\nolimits_{\kern -5pt{#1}}}
          {\mathop{\vrule width 5pt height 3 pt depth -2.6pt
                  \kern -6pt \intop}\nolimits_{\kern -3pt{#1}}}
          {\mathop{\vrule width 5pt height 3 pt depth -2.6pt
                  \kern -6pt \intop}\nolimits_{\kern -3pt{#1}}}
          {\mathop{\vrule width 5pt height 3 pt depth -2.6pt
                  \kern -6pt \intop}\nolimits_{\kern -3pt{#1}}}}
\newcommand{\pip}{\varphi}
\newcommand{\eps}{\varepsilon}
\newcommand{\loc}{{\mbox{\scriptsize{loc}}}}

\newcommand{\BV}{\mathrm{BV}}

\DeclareMathOperator{\Lip}{Lip}

\DeclareMathOperator*{\aplim}{\mathrm{ap\,lim}}

\begin{document}
\title{Stability and continuity of functions of least gradient
\thanks{{\bf 2000 Mathematics Subject Classification}: Primary 26B30; Secondary 31E99, 
31C45, 26B15. \hfill \break {\it Keywords\,}: least gradient, BV, metric measure space, approximate continuity, continuity, stability, jump set, Dirichlet problem, minimal surface.}
}
\author{H.~Hakkarainen, R.~Korte, P.~Lahti and N.~Shan\-mugalingam
\footnote{R. K. was  supported by Academy of Finland, grant \#250403. P. L. and H. H. were supported
by the Finnish Academy of Science and Letters, the Vilho, Yrj\"o and Kalle V\"ais\"al\"a Foundation.
N.S.~was partially supported by the NSF
grant~DMS-1200915.  Part of this research was
conducted while the authors visited Institut Mittag-Leffler, Sweden; they wish to thank this
institution for its kind hospitality. The authors also wish to thank Juha Kinnunen for his kind encouragement in this project and 
Michele Miranda Jr.~for useful discussion regarding Proposition~3.1.}}
\maketitle

\begin{abstract} 
In this note we prove that on metric measure spaces, functions of least gradient, as well as local minimizers of the area functional (after modification on a set of measure zero) are continuous
everywhere outside their jump sets. As a tool, 
we develop some stability properties of sequences of least gradient functions. We also apply these tools
to prove a maximum principle for functions of least gradient that arise as solutions to a Dirichlet problem.
\end{abstract}

\section{Introduction}

The theory of minimal surfaces in the Euclidean setting has been studied extensively, for example, in~\cite{Ah}, \cite{R},
\cite{BDeGG}, \cite{Giu84}, \cite{Sim}, \cite{Sim2}, \cite{SWZ} from the point of view of regularity. The literature on this
subject is extensive, and it is impossible to list all references; only a small sampling is given here.
Much of this study had been
in the direction of understanding the regularity of minimal surfaces obtained (locally) as graphs of functions. 
However, the work of \cite{BDeGG}, \cite{SWZ} and~\cite{Z} and others studies more general ``least gradient" functions
and their regularity, and it is shown in~\cite{SWZ} that if boundary data is Lipschitz continuous and the (Euclidean)
boundary of a domain has positive mean curvature, then the least gradient solution to the corresponding Dirichlet
problem is locally Lipschitz continuous in the domain. However, such Lipschitz regularity has been shown to fail even
in a simple weighted Euclidean setting, see~\cite{HKL}, where an example is given of a solution with jump discontinuities in the domain, even though the boundary data is Lipschitz continuous.
Therefore, in a more general setting, it is
natural to ask whether functions of least gradient are continuous
outside their jump sets. The principal goal
of this note is to show that every function of least gradient is necessarily continuous outside its jump set, even when
the boundary data is not continuous. 

The setting we consider here is that of a complete metric measure space $X=(X,d,\mu)$ equipped with a doubling 
Borel regular outer measure $\mu$ supporting a $(1,1)$-Poincar\'e inequality. We consider functions of bounded
variation in the sense of~\cite{A2}, \cite{M}, and~\cite{AMP}, and  functions of least gradient
in a domain $\Omega\subset X$. 

In considering regularity properties of functions of least gradient, we need some tools related to stability properties of
least gradient function families. Therefore we extend the study to also include 
questions related to stability properties of least gradient functions (minimizers) and 
quasiminimizers. We show that being a function of least gradient
is a property preserved under $L^1_{\loc}(\Omega)$-convergence. 
We then obtain partial regularity results for functions of least gradient. 
Namely, we show that such minimizers are continuous at points of approximate continuity, that is, away from the
jump discontinuities of the function. Observe that by the results of~\cite{AMP}, the jump set of a BV function has
$\sigma$-finite co-dimension $1$ Hausdorff measure; hence there is a plenitude of points where the least gradient
function is continuous.
 
As a further application of the tools developed to study the above regularity, we  
obtain a maximum principle for least gradient functions obtained as solutions to a Dirichlet problem.

In tandem with the development of least gradient theory in the metric setting, the papers~\cite{HKL} and~\cite{HKLL}
develop the existence and trace theory of minimizers of functionals of linear growth in the metric setting (analogously
to the problem considered by Giusti in~\cite{Giu84}). In the case that the function $f$ of linear growth satisfies
$f(0)=0$ or $\liminf_{t\to 0^+}(f(t)-f(0))/t>0$, the results of this note can be adapted to study the regularity
and maximum principle properties of the associated minimizers, but mostly we limit ourselves to the least gradient theory.
However, in the case of the area functional $f(t)=\sqrt{1+t^2}$, we give an explicit proof that minimizers of this functional are continuous at points of approximate continuity.

This paper is organized as follows. In Section~2 we introduce the concepts and background needed for our study, and in Section~3 we consider the stability problem for least gradient functions on a given
domain. In Section~4 we use the tools developed in Section~3 to show that least gradient functions are continuous
everywhere outside their jump sets. In Section 5 we give a further application of the tools
 developed in Section~3 to prove a maximum principle for functions of least gradient that arise as solutions to a Dirichlet problem. In the final section of this paper we extend the continuity result of Section 4 to the case of the area functional.

\section{Notation and background}

In this section we
introduce the notation and the problems we consider.

In this paper, $(X,d,\mu)$ is a complete metric space equipped
with a metric $d$ and a Borel regular outer measure $\mu$.
The measure is assumed to be doubling, meaning that there exists a constant $c_d>0$ such that
\[
0<\mu(B(x,2r))\leq c_d\mu(B(x,r))<\infty
\]
for every ball $B(x,r):=\{y\in X:\,d(y,x)<r\}$.

We say that a property holds for almost every $x\in X$, or a.e. $x\in X$,  if there is a set $A\subset X$ with $\mu(A)=0$ and the property holds outside $A$. We will use the letter $C$ to denote a positive constant whose value is not
necessarily the same at each occurrence.

We recall that a complete metric space endowed with a doubling measure is proper,
that is, closed and bounded sets are compact. Since $X$ is proper, for any open set $\Omega\subset X$
we define e.g. $\textrm{Lip}_{\loc}(\Omega)$ as the space of
functions that are Lipschitz in every $\Omega'\Subset\Omega$.
The notation $\Omega'\Subset\Omega$ means that $\Omega'$ is open and that $\overline{\Omega'}$ is a
compact subset of $\Omega$.

For any set $A\subset X$ and $0<R<\infty$, the restricted spherical Hausdorff content
of codimension $1$ is defined as
\[
\mathcal{H}_{R}(A):=\inf\left\{ \sum_{i=1}^{\infty}\frac{\mu(B(x_{i},r_{i}))}{r_{i}}:\,A\subset\bigcup_{i=1}^{\infty}B(x_{i},r_{i}),\,r_{i}\le R\right\}.
\]
The Hausdorff measure of codimension $1$ of a set
$A\subset X$ is
\[
\mathcal{H}(A):=\lim_{R\rightarrow0}\mathcal{H}_{R}(A).
\]


A curve is a rectifiable continuous mapping from a compact interval
to $X$, and is usually denoted by the symbol $\gamma$.
A nonnegative Borel function $g$ on $X$ is an upper gradient 
of an extended real-valued function $u$
on $X$ if for all curves $\gamma$ on $X$, we have
\[
|u(x)-u(y)|\le \int_\gamma g\,ds
\]
whenever both $u(x)$ and $u(y)$ are finite, and 
$\int_\gamma g\, ds=\infty $ otherwise.
Here $x$ and $y$ are the end points of $\gamma$.
Given a locally Lipschitz continuous function $u$ on $\Omega$, we define the local Lipschitz constant of $u$ as
\begin{equation}\label{eq:definition of local lipschitz constant}
 \Lip u(x):=\limsup_{r\to0^+}\sup_{y\in B(x,r)\setminus\{x\}} \frac{|u(y)-u(x)|}{d(y,x)}.
\end{equation}
In particular, $\Lip u$ is known to be an upper gradient of $u$, see e.g. \cite[Proposition 1.11]{C}.

We consider the following norm
\[
\Vert u\Vert_{N^{1,1}(X)}:=\Vert u\Vert_{L^1(X)}+\inf_g\Vert g\Vert_{L^1(X)},
\]
with the infimum taken over all upper gradients $g$ of $u$. 
The Newton-Sobolev, or Newtonian space is defined as
\[
N^{1,1}(X):=\{u:\|u\|_{N^{1,1}(X)}<\infty\}/{\sim},
\]
where the equivalence relation $\sim$ is given by $u\sim v$ if and only if 
\[
\Vert u-v\Vert_{N^{1,1}(X)}=0.
\]
Similarly, we can define $N^{1,1}(\Omega)$ for an open set $\Omega\subset X$. For more on Newtonian spaces, we refer to \cite{S} or \cite{BB}.

Next we recall the definition and basic properties of functions
of bounded variation on metric spaces.
A good discussion of BV functions in the Euclidean setting can be
found in~\cite{Zie89} and~\cite{AmbFP00}.
In the metric setting, the corresponding theory was first studied by
Miranda Jr.~in~\cite{M}, and further developed in~\cite{A1}, \cite{A2}, \cite{AMP}, and~\cite{KKST}.
For $u\in L^1_{\text{loc}}(X)$, we define the total variation of $u$ as
\[
\|Du\|(X):=\inf\Big\{\liminf_{i\to\infty}\int_X g_{u_i}\,d\mu:\,N^{1,1}_{\loc}(X)\ni u_i\to u\textrm{ in } L^1_{\text{loc}}(X)\Big\},
\]
where $g_{u_i}$ is an upper gradient of $u_i$.
In the literature, the upper gradient is sometimes replaced by a local Lipschitz constant and the approximating functions $u_i$ are sometimes required to be locally Lipschitz, but all the definitions give the same result according to \cite[Theorem 1.1]{ADiMa}.
We say that a function $u\in L^1(X)$ is of bounded variation, 
and denote $u\in\BV(X)$, if $\|Du\|(X)<\infty$. 
Moreover, a $\mu$-measurable set $E\subset X$ is said to be of finite perimeter if $\|D\chi_E\|(X)<\infty$.
By replacing $X$ with an open set $\Omega\subset X$ in the definition of the total variation, we can define $\|Du\|(\Omega)$.
The $\BV$ norm is given by
\[
\Vert u\Vert_{\BV(\Omega)}:=\Vert u\Vert_{L^1(\Omega)}+\Vert Du\Vert(\Omega).
\]
For an arbitrary set $A\subset X$, we define
\[
\|Du\|(A):=\inf\bigl\{\|Du\|(\Omega):\,\Omega\supset A,\,\Omega\subset X
\text{ is open}\bigr\}.
\]
If $u\in\BV(\Omega)$, $\|Du\|(\cdot)$ Radon measure of finite mass in $\Omega$ \cite[Theorem 3.4]{M}.

We also denote the perimeter of $E$ in $\Omega$ by
\[
P(E,\Omega):=\|D\chi_E\|(\Omega).
\]
We have the following coarea formula given by Miranda in \cite[Proposition 4.2]{M}: if $F\subset X$ is a Borel set 
and $u\in \BV(X)$, we have
\begin{equation}\label{eq:coarea}
\|Du\|(F)=\int_{-\infty}^{\infty}P(\{u>t\},F)\,dt.
\end{equation}

The approximate upper and lower limits of an extended real-valued function $u$ on $X$ are:
\begin{align}\label{eq:appcont}
u^{\vee}(x)&:=\inf\left\{t\in\R:\,\lim\limits_{r\to 0^+}\frac{\mu\left(\{u\leq t\}\cap B(x,r)\right)}{\mu\left(B(x,r)\right)}=1\right\}, \\
u^{\wedge}(x)&:=\sup\left\{t\in\R:\,\lim\limits_{r\to 0^+}\frac{\mu\left(\{u\geq t\}\cap B(x,r)\right)}{\mu\left(B(x,r)\right)}=1\right\}\notag.
\end{align}
If $u^{\vee}(x)=u^{\wedge}(x)$, we denote by
\[
\aplim\limits_{y\to x} u(y):=u^{\vee}(x)=u^{\wedge}(x)
\]
the \textit{approximate limit} of function $u$ at $x\in X$. The function $u$ is \textit{approximately continuous} at $x\in X$ if
\[
\aplim\limits_{y\to x} u(y)=u(x).
\]

We assume that the space $X$ supports a $(1,1)$-Poincar\'e inequality,
meaning that for some constants $c_P>0$ and $\lambda \ge 1$, for every
ball $B(x,r)$, for every locally integrable function $u$,
and for every upper gradient $g$ of $u$, we have 
\[
\vint{B(x,r)}|u-u_{B(x,r)}|\, d\mu 
\le c_P r\,\vint{B(x,\lambda r)}g\,d\mu,
\]
where 
\[
u_{B(x,r)}:=\vint{B(x,r)}u\,d\mu :=\frac 1{\mu(B(x,r))}\int_{B(x,r)}u\,d\mu.
\]

For BV functions, we get the following version of the $(1,1)$-Poincar\'e inequality. Given any locally integrable function $u$, by applying the $(1,1)$-Poincar\'e inequality to an approximating sequence of functions, we get
\[
\vint{B(x,r)}|u-u_{B(x,r)}|\, d\mu 
\le c_P r\,\frac{\Vert Du\Vert (B(x,\lambda r))}{\mu(B(x,\lambda r))}.
\]



\begin{definition}
We denote by $BV_c(\Omega)$ the functions $g\in BV(X)$ 
with compact support in $\Omega$, and by $BV_0(\Omega)$ the functions $g\in BV(X)$ such that $g=0$ $\,\mu$-a.e.~in 
$X\setminus\Omega$. 
\end{definition}

We next recall the various notions of minimizers; see also e.g.~\cite{BDeGG}.

\begin{definition}
A function $u\in BV_{\loc}(\Omega)$ is of \emph{least gradient} in $\Omega$ if whenever $g\in BV_c(\Omega)$ with 
$K=\text{supt}(g)$, we have
 \[
  \Vert Du\Vert (K)\le \Vert D(u+g)\Vert(K).
 \] 
 Following~\cite{KKLS}, we say that a set $E\subset X$ is \emph{of minimal surface in $\Omega$} if 
 $\chi_E\in BV_{\loc}(\Omega)$ and
\[
\Vert D\chi_E\Vert(K)\le \Vert D\chi_F\Vert(K)
\]
for all $\mu$-measurable sets $F\subset X$ for which $K:=\overline{F\Delta E}$ is a compact subset of $\Omega$.
A function $u\in BV_{\loc}(\Omega)$ is of \emph{$Q$-quasi least gradient} in $\Omega$ for $Q\ge 1$ if
 \[
  \Vert Du\Vert (K)\le Q\, \Vert D(u+g)\Vert(K)
 \] 
 for all $g\in BV_c(\Omega)$ and $K=\text{supt}(g)$.
\end{definition}

In our definition of sets of minimal surface, observe that the sets $F$ are precisely ones that can be written as
$(E\cup\widehat{F})\setminus\widehat{G}$ for some relatively compact $\mu$-measurable subsets
$\widehat{F},\widehat{G}$ of $\Omega$. Thus our definition is consistent with \cite{KKLS}. Observe also that by the definition and measure property of the total variation, the support of the perturbation $K$ can always be replaced by any larger relatively compact subset of $\Omega$ in the above definitions.

We also consider the corresponding Dirichlet problem in the metric setting. Recall that in the Euclidean setting with a 
Lipschitz domain $\Omega$, the least gradient problem with a given boundary datum $f$ can be stated as
\[
\min\left\{\Vert Du\Vert(\Omega): \,u\in\BV(\Omega),\,u=f\,\text{ on }\,\partial\Omega\right\}.
\]
If we do not require, for instance, the continuity of the boundary data, then the boundary value has to be understood in 
a suitable sense, e.g. as a trace of a BV function. One possibility is to consider a relaxed problem with a penalization term
\[
\Vert Du\Vert(\Omega)+\int_{\partial\Omega}|Tu-Tf|\,d\mathcal{H}^{n-1},
\] 
where $Tu$ and $Tf$ are traces of $u$ and $f$ on $\partial\Omega$, provided such traces exist, and $\mathcal H^{n-1}$ is the $n-1$-dimensional Hausdorff measure.
In order to avoid working directly with traces of BV functions, we consider an equivalent problem formulated in a 
larger domain, namely $X$. 

\begin{definition}\label{def_leastgrad_dir}
Given an open set $\Omega\subset X$ and $f\in BV(X)$, we say that $u\in BV(X)$ is a \emph{solution to the Dirichlet problem for least gradients in $\Omega$
 with boundary data $f$}, if $u-f\in BV_0(\Omega)$ and
 whenever $g\in BV_0(\Omega)$, we have
 \[
   \Vert Du\Vert(\overline{\Omega})\le \Vert D(u+g)\Vert(\overline{\Omega}).
 \]
\end{definition}

Note that such a solution is a function of least gradient and that this problem is the same as minimizing 
$\Vert Du\Vert(\overline{\Omega})$ over all $u\in BV(X)$ that satisfy $u-f\in BV_0(\Omega)$. Furthermore, since we have
$\Vert Du\Vert(X)=\Vert Du\Vert(\overline{\Omega})+\Vert Df\Vert(X\setminus\overline{\Omega})$, the problem is 
equivalent to, and the minimizers are the same as, when minimizing
$\Vert Du\Vert(X)$ over all $u\in BV(X)$ that satisfy $u-f\in BV_0(\Omega)$. Thus the problem we consider here is of
the same type as in \cite{HKL}, \cite{HKLL}.

\section{Stability of least gradient function families}

To answer questions regarding continuity properties of functions of least gradient (outside their jump sets),
it turns out that tools related to the stability of least gradient function families are needed. In this section we therefore study such stability properties.

The following stability result is key in the study of continuity properties of functions of least gradient. 
In the Euclidean setting, the proof of this result found in~\cite[Theorema 3]{M-Sr} is 
based on trace theorems for BV functions. Given the lack of trace theorems in the
metric setting, the proof given here is different from that of~\cite{M-Sr}, but the philosophy underlying the
proof is the same. We thank Michele Miranda for explaining the proof in~\cite{M-Sr} (which is in Italian) and for 
suggesting a way to modify it.

\begin{proposition}\label{stab-min}
Let $\Omega\subset X$ be an open set, let $u_k\in BV_{loc}(\Omega)$, $k\in\N$, be a sequence of functions
of least gradient in $\Omega$, and suppose that $u$ is a function in $\Omega$ such that 
$u_k\to u$ in $L^1_{loc}(\Omega)$. Then $u$ is a function of least gradient in $\Omega$.
\end{proposition}

An analogous statement holds for sequences of $Q$-quasi least gradient functions; the $L^1_{\loc}(\Omega)$-limits
of such sequences are also $Q$-quasi least gradient functions. The proof is mutatis mutandis the same as the proof of the
above proposition; we leave the interested reader to verify this.

To prove the above proposition, we need the following version of the product rule (Leibniz rule) for functions of bounded
variation.
\begin{lemma}\label{lem:Leibniz rule}
If $\Omega$ is an open set, $u,v\in BV(\Omega)$ and $\eta\in \Lip(\Omega)$, then $\eta u+(1-\eta)v\in BV(\Omega)$ with 
\[
  d\Vert D(\eta u+(1-\eta)v)\Vert\le \eta \,d\Vert Du\Vert+(1-\eta)\, d\Vert Dv\Vert+|u-v|g_{\eta}\, d\mu,
\]
where $g_{\eta}$ is any bounded upper gradient of $\eta$.
\end{lemma}
\begin{proof}
According to \cite[Lemma 6.2.1]{Cam08} or \cite[Proposition 3.7]{HKLL}, we can pick sequences $u_i\in N^{1,1}_{\loc}(\Omega)$ and $v_i\in N^{1,1}_{\loc}(\Omega)$ such that $u_i\to u$ in $L^1_{\loc}(\Omega)$, $v_i\to v$ in $L^1_{\loc}(\Omega)$, and $g_{u_i}\,d\mu\to d\Vert Du\Vert$, $g_{v_i}\,d\mu\to d\Vert Dv\Vert$ weakly* in $\Omega$, where $g_{u_i}$ and $g_{v_i}$ are upper gradients of $u_i$ and $v_i$, respectively. Then $\eta u_i+(1-\eta)v_i\in N^{1,1}_{\loc}(\Omega)$, $\eta u_i+(1-\eta)v_i\to \eta u+(1-\eta)v$ in $L^1_{\loc}(\Omega)$, and every $\eta u_i+(1-\eta)v_i$ has an upper gradient
\[
\eta g_{u_i}+(1-\eta) g_{v_i}+|u_i-v_i|g_{\eta},
\]
see \cite[Lemma 2.18]{BB}. Now take any open sets $U'\Subset U\subset\Omega$, and any $\psi\in\Lip_c(U)$ with $0\le\psi\le 1$ and $\psi=1$ in $U'$. Then we have
\begin{align*}
\Vert D(\eta u+(1-&\eta)v)\Vert(U')\le\liminf_{i\to\infty}\int_{U'}\eta g_{u_i}+(1-\eta) g_{v_i}+|u_i-v_i|g_{\eta}\,d\mu\\
      &\le\liminf_{i\to\infty}\int_{U}\psi(\eta g_{u_i}+(1-\eta) g_{v_i}+|u_i-v_i|g_{\eta})\,d\mu\\
      &= \int_{U}\psi \eta\,d\Vert Du\Vert+\int_{U}\psi(1-\eta)\,d\Vert Dv\Vert+\int_{U}\psi|u-v|g_{\eta}\,d\mu\\
      &\le \int_{U} \eta\,d\Vert Du\Vert+\int_{U}(1-\eta)\,d\Vert Dv\Vert+\int_{U}|u-v|g_{\eta}\,d\mu.
\end{align*}
By the fact that $U'\Subset U$ was arbitrary and by the inner regularity of the total variation, see the proof of \cite[Theorem 3.4]{M}, this implies that
\[
\Vert D(\eta u+(1-\eta)v)\Vert(U)\le \int_{U}\eta \,d\Vert Du\Vert+\int_{U}(1-\eta) \,d\Vert Dv\Vert+\int_{U}|u-v|g_{\eta}\,d\mu.
\]
Since the variation measure of arbitrary sets is defined by approximation with open sets, we have the result.
\end{proof}

We will also need the following inequality from~\cite[Inequality~(4.3)]{KKLS} for functions of least gradient in $\Omega$.
Whenever $B=B(x,R)\Subset\Omega$, we have the De Giorgi inequality for every $0<r<R$:
\begin{equation}\label{eq:De Giorgi inequality}
  \Vert Du\Vert (B(x,r))\le \frac{C}{R-r}\int_{B(x,R)}|u|\, d\mu.
\end{equation}
Observe that in \cite{KKLS}, this is proved when $u$ is the characteristic function of a set, but the proof works for more general functions as well.

\begin{proof}[Proof of Proposition~\ref{stab-min}]
Take any compact set $K_1\subset \Omega$. Because $u_k\to u$ in $L^1_{\loc}(\Omega)$, we know that the functions $u_k$ are bounded in $L^1(K_2)$ for 
compact sets $K_2\subset \Omega$. Hence 
by covering $K_1$ by balls of radius $r$
so that concentric balls 
of radius $2r$ are relatively compact subsets of $\Omega$, and by then applying
the above De Giorgi inequality, we see that
\begin{equation}\label{eq:1}
  \sup_k\Vert Du_k\Vert(K_1)\le C_{K_1}<\infty.
\end{equation}
By the fact that $u_k\to u$ in $L^1_{\loc}(\Omega)$ and by the lower semicontinuity of the total variation, we can conclude that $u\in BV_{\loc}(\Omega)$.
To show that $u$ is of least gradient in $\Omega$, we fix a function $g\in BV(\Omega)$ such that the support $\widetilde{K}$ of
$g$ is a compact subset of $\Omega$. We need to show that
\[
  \Vert Du\Vert(\widetilde{K})\le \Vert D(u+g)\Vert(\widetilde{K}).
\]

By~\eqref{eq:1} we know that the sequence of Radon measures $\Vert Du_k\Vert$ is locally bounded in $\Omega$. Hence a diagonalization argument gives a subsequence, also denoted by $\Vert Du_k\Vert$, and a
Radon measure $\nu$ on $\Omega$, such that $\Vert Du_k\Vert\to\nu$ weakly$^*$ in $\Omega$. 

Let $K\subset\Omega$ be a compact set such that $\widetilde{K}\subset K$ and
\[
\nu(\partial K)=0, 
\]
and let $\eps>0$  such that $K_\eps:=\bigcup_{x\in K}B(x,\eps)\Subset\Omega$. 
We choose a Lipschitz function $\eta$ on $X$ such that $0\le \eta\le 1$ on $X$, $\eta=1$ in $K$, and
$\eta=0$ in $X\setminus K_{ \epsilon/2}$, and for each positive integer $k$ we set
\[
  g_k:=\eta(u+g)+(1-\eta)u_k.
\]
Note that $g_k=u_k$ on $\Omega\setminus K_{\eps/2}$, and so by the lower
semicontinuity of the total variation and the fact that
$u_k\to u$ in $L^1(K_\eps)$, the minimality of $u_k$, and the Leibniz rule of Lemma \ref{lem:Leibniz rule}, we have
\begin{align*}
 \Vert Du\Vert(K_\eps)&\le \liminf_{k\to\infty} \Vert Du_k\Vert(K_\eps)\\
        &\le \liminf_{k\to\infty}\Vert Dg_k\Vert(K_\eps)\\
       &\le \Vert D(u+g)\Vert(K_\eps)\\
               &\qquad\qquad+\liminf_{k\to\infty}\left[\Vert Du_k\Vert(K_\eps\setminus K)+C_\eta\int_{K_\eps\setminus K}|u-u_k|\, d\mu\right]\\
        &=  \Vert D(u+g)\Vert(K_\eps)+\liminf_{k\to\infty}\Vert Du_k\Vert(K_\eps\setminus K)\\
        &\le \Vert D(u+g)\Vert(K_\eps)+\nu(\overline{K_\eps\setminus K}).
\end{align*}
Letting $\eps\to 0$, we obtain
\[
   \Vert Du\Vert(K)\le \Vert D(u+g)\Vert(K)+\nu(\partial K)=\Vert D(u+g)\Vert(K).
\]
Since $K$ contains the support of $g$, we conclude that $u$ is of least gradient in $\Omega$.
\end{proof}

While the above proposition does not require the functions $u_k$ to be in the global space $BV(X)$,
the next stability result considers what happens when each $u_k\in BV(X)$ is a solution to the Dirichlet problem
for least gradients with boundary data in $BV(X)$.

\begin{proposition}
Let $\Omega$ be a bounded open set in $X$ such that $\mu(X\setminus\Omega)>0$.
Take a sequence of functions $f_k\in BV(X)$, $k\in\N$, and suppose that each $u_k\in BV(X)$ is a solution to
the Dirichlet problem for least gradients in $\Omega$ with boundary data $f_k$. 
Suppose also
that $f_k\to f$ in $BV(X)$ {\rm (}that is, $\Vert f-f_k\Vert_{L^1(X)}+\Vert D(f_k-f)\Vert(X)\to 0$ as $k\to\infty${\rm )}.
Then there is a function $u\in BV(X)$ 
such that a subsequence of $u_{k}$ converges to  $u$ in $L^{1}(\Omega)$, and 
$u$ is a solution to the Dirichlet problem for least gradients 
in $\Omega$ with boundary data $f$.
\end{proposition}

\begin{proof}
Let $B\Supset\Omega$ be a ball such that $\mu(B\setminus\Omega)>0$.
By the $(1,1)$-Poincar\'e inequality, if $v\in BV_0(\Omega)$, then
\[
 \int_{\Omega}|v|\, d\mu\le C_0\, \Vert Dv\Vert(\overline{\Omega}),
\]
where $C_0$ depends only on the radius of $B$, the Poincar\'e inequality constants, the doubling constant
of $\mu$, and the ratio $\mu(B\setminus\Omega)/\mu(B)$ \cite[Lemma 2.2]{KKLS}. 

By definition, we know that $u_k-f_k\in BV_0(\Omega)$, and hence for each positive integer $k$,
\begin{align*}
  \int_X  |u_k-f_k|\, d\mu+ \Vert D(u_k-f_k)\Vert(X) 
    &=\int_\Omega  |u_k-f_k|\, d\mu+ \Vert D(u_k-f_k)\Vert(\overline{\Omega})\\
    &\le [C_0+1]\, \Vert D(u_k-f_k)\Vert(\overline{\Omega})\\
    &\le [C_0+1]( \Vert Du_k\Vert(\overline{\Omega})+ \Vert Df_k\Vert(\overline{\Omega}))\\
    &\le 2[C_0+1]\, \Vert Df_k\Vert(\overline{\Omega}).
\end{align*}
The last inequality follows from the fact that $u_k$ is a solution to the Dirichlet
problem for least gradients with boundary data $f_k$. It follows that the sequence $\{u_k-f_k\}_k$ is a bounded sequence in
$BV_0(\Omega)\subset BV(X)$, and hence by the compact embedding theorem for $BV(X)$ (see~\cite[Theorem~3.7]{M}),
there is a subsequence, also denoted by $\{u_k-f_k\}_k$, and a function $v\in BV_0(\Omega)$, such that
$u_k-f_k\to v$ in $L^1_{\loc}(X)$, and so by the compactness of $\overline{\Omega}$,
we know that $u_k-f_k\to v$ in $L^1(\Omega)$. Hence $u_k\to f+v=:u$ in $L^1(\Omega)$,
with $u-f=v\in BV_0(\Omega)$.

By the lower semicontinuity of the total variation, we have
\begin{align*}
  \Vert Du\Vert(\overline{\Omega})+\Vert Df\Vert(X\setminus\overline{\Omega})
      &=\Vert Du\Vert (X)\\
      &\le \liminf_{k\to\infty} \Vert Du_k\Vert(X)\\
      &=\liminf_{k\to\infty}\left(\Vert Du_k\Vert(\overline{\Omega})+\Vert Df_k\Vert(X\setminus\overline{\Omega})\right)\\
      &=\liminf_{k\to\infty}\Vert Du_k\Vert(\overline{\Omega})+\Vert Df\Vert(X\setminus\overline{\Omega}),
\end{align*}
so that
\begin{equation}\label{eq:lower semicontinuity in closed set}
\Vert Du\Vert(\overline{\Omega})\le \liminf_{k\to\infty}\Vert Du_k\Vert(\overline{\Omega}).
\end{equation}
Now let $g\in BV_0(\Omega)$, and $h:=f+g$.
Furthermore, let $\widehat{f_k}:=f_k+g$.
Since each $u_k$ is a solution to the Dirichlet problem for least gradients with boundary data $f_k$,
we have $\Vert Du_k\Vert(\overline{\Omega})\le \Vert D\widehat{f_k}\Vert(\overline{\Omega})$. 
Moreover, $\widehat{f_k}-h=f_k-f\to 0$ in $BV(X)$ as $k\to\infty$.

By combining \eqref{eq:lower semicontinuity in closed set} with these facts, we get
\[
  \Vert Du\Vert(\overline{\Omega})\le \liminf_{k\to\infty}\Vert Du_k\Vert(\overline{\Omega})
     \le \liminf_{k\to\infty}\Vert D\widehat{f_k}\Vert(\overline{\Omega})=\Vert Dh\Vert(\overline{\Omega}).
\]
Thus $u$ is a solution to the Dirichlet problem with boundary
data $f$. This concludes the proof. 
\end{proof}


%

The above two stability results require the sequence $u_k$ to converge in $L^1_{\loc}(\Omega)$ (while the
second stability result above did not explicitly require this, it was an almost immediate consequence of the
hypothesis). The next proposition considers the weakest form of stability, namely, what happens when the sequence
$u_k$ is only known to converge pointwise almost everywhere to a function $u$ in $\Omega$.

Recall that given an extended real-valued function $u$ on $\Omega$, its super-level sets are sets of
the form $\{x\in\Omega\, :\, u(x)>t\}$ for $t\in \R$.

\begin{proposition}\label{pointws}
Let $\{u_k\}$ be a sequence of functions of least gradient in an open set $\Omega$, and let $u$ be a measurable
function on $\Omega$, finite-valued $\mu$-almost everywhere, such that 
$u_k\to u$ $\mu$-a.e.~in $\Omega$. Then the characteristic functions of the super-level sets $\{x\in\Omega\, :\, u(x)>t\}$ are functions of least gradient in $\Omega$ for almost every $t\in\R$. If in addition $\sup_{k\in\N,x\in K}|u_k(x)|=: M_K<\infty$ for every compact set $K\subset\Omega$, 
or if $\sup_k\Vert Du_k\Vert(K)=:C_K<\infty$ for every compact set $K\subset\Omega$,
then $u$ is of
least gradient in $\Omega$ and there is a subsequence of $\{u_k\}$ that converges in $L^1_{\loc}(\Omega)$ to
$u$. 
\end{proposition}

To prove this proposition, we need the following two lemmas, which will also be quite useful in the study of
continuity properties of minimizers undertaken in the next section. The argument in the lemmas is based on 
Bombieri--De Giorgi--Giusti~\cite{BDeGG}.

\begin{lemma}\label{lem:decomposition of BV function by truncation}
Let $u\in\BV(X)$, and for a given $t\in\R$, let $u_1:=\min\{u,t\}$ and $u_2:=(u-t)_+$, so that $u=u_1+u_2$.
Then we have
\[
  \Vert Du_1\Vert+\Vert Du_2\Vert=\Vert Du\Vert
\]
in the sense of measures.
\end{lemma}

\begin{proof}
For any open set $G\subset X$, we have by the coarea formula~\eqref{eq:coarea} 
and by the fact that $P(G,G)=0=P(\emptyset,G)$ for open $G$,
\begin{align*}
 \Vert Du_1\Vert(G)&=\int_{-\infty}^\infty P(\{x\in G\, :\, u_1(x)>s\}, G)\, ds\\
    &=\int_{-\infty}^t P(\{x\in G\, :\, u_1(x)>s\}, G)\, ds\\
    &=\int_{-\infty}^t P(\{x\in G\, :\, u(x)>s\}, G)\, ds.
\end{align*}
Similarly, we have
\begin{align*}
 \Vert Du_2\Vert(G)&=\int_{-\infty}^\infty P(\{x\in G\, :\, u_2(x)>s\}, G)\, ds\\
   &=\int_0^\infty P(\{x\in G\, :\, u_2(x)>s\}, G)\, ds\\
   &=\int_0^\infty P(\{x\in G\, :\, u(x)-t>s\}, G)\, ds\\
   &=\int_t^\infty P(\{x\in G\, :\, u(x)>s\}, G)\, ds.\\
\end{align*}
Therefore we have by~\eqref{eq:coarea} again,
\[
 \Vert Du_1\Vert(G)+\Vert Du_2\Vert(G)=\int_{-\infty}^\infty P(\{x\in G\, :\, u(x)>s\},G)\, ds
    =\Vert Du\Vert(G),
\]
and since the variation measure of general sets is defined by approximation with open sets, we can conclude that 
$\Vert Du_1\Vert+\Vert Du_2\Vert=\Vert Du\Vert$.
\end{proof}

\begin{lemma}\label{min-surface-levels}
Let $\Omega\subset X$ be an open set, and suppose that $u$ is a function of least gradient in $\Omega$. Then for each 
$t\in\R$, the characteristic function $\chi_{E_t}$ of the super-level set
\[
 E_t:=\{x\in\Omega\, :\, u(x)>t\}
\]
is a function of least gradient in $\Omega$. 
\end{lemma}

\begin{proof}
Take any $t\in\R$ and let $u_1,u_2$ be defined as in the previous lemma. Let $g\in BV_c(\Omega)$ with $K:=\text{supt}(g)$. 
Then we have by the previous lemma, the minimality of $u$, and the subadditivity of the total variation, 
\begin{align*}
   \Vert Du_1\Vert(K)+\Vert Du_2\Vert(K)=\Vert Du\Vert(K)&\le \Vert D(u+g)\Vert(K)\\
       &=\Vert D(u_1+u_2+g)\Vert(K)\\
       &\le \Vert D(u_1+g)\Vert(K)+\Vert Du_2\Vert(K).
\end{align*}
It follows that 
\[
  \Vert D u_1\Vert(K)\le \Vert D(u_1+g)\Vert(K),
\]
so that $u_1$ is also of least gradient in $\Omega$.
Mutatis mutandis we can show that $u_2$ is also of least gradient in $\Omega$. Hence we have that whenever 
$\eps>0$, the function 
\[
  u_{t,\eps}:=\frac1\eps\, \min\{\eps,\, (u-t)_+\}
\]
is of least gradient in $\Omega$. 
By the Lebesgue dominated convergence theorem, for any compact $K\subset\Omega$ it is true that
\[
  \int_K|u_{t,\eps}-\chi_{E_t}|\, d\mu=\int_{\{x\in K\, :\, 0<u(x)-t\le \eps\}}\left(1-\frac{u(x)-t}{\eps}\right)\, d\mu(x)
     \to 0
\]
as $\eps\to 0$.
Hence $u_{t,\eps}\to\chi_{E_t}$ in $L^1_{\loc}(\Omega)$ as $\eps\to 0$, 
from which, together with Proposition~\ref{stab-min},  we 
conclude that $\chi_{E_t}$ is of least gradient in $\Omega$.
This completes the proof of the lemma.
\end{proof}

\begin{proof}[Proof of Proposition~\ref{pointws}]
Since $u_k\to u$ almost everywhere in $\Omega$, it follows that for almost every $t\in\R$ we have
$\chi_{\{x\in\Omega\, :\, u_k(x)>t\}}\to\chi_{\{x\in\Omega\, :\, u(x)>t\}}$ almost everywhere in 
$\Omega$. Indeed, to see this, we set $N$ to be the collection of points $x\in\Omega$ such that
$u_k(x)$ does not converge to $u(x)$. Then $\mu(N)=0$. For any $t\in\R$ we have that if 
$u_k(x)\le t$ for a subsequence of $k$ but $u(x)>t$, then $x\in N$. So, setting
\[
  K_t:=\{x\in\Omega\setminus N\, :\, u_k(x)>t\text{ for a subsequence of }k\text{ but }u(x)\le t\},
\]
for each $x\in\Omega\setminus (K_t\cup N)$ we see that
$\chi_{\{y\in\Omega\, :\, u_k(y)>t\}}(x)\to\chi_{\{y\in\Omega\, :\, u(y)>t\}}(x)$. Note that
for $x\in K_t$, we have $u(x)=t$. Therefore when $s\ne t$ we have that $K_s\cap K_t$ is empty.
Thus the family $\{K_t\}_{t\in\R}$ is pairwise disjoint, and hence by the local finiteness of $\mu$ there is
at most a countable number of $t\in\R$ for which $\mu(K_t)>0$. We conclude that for almost every
$t\in\R$, $\chi_{\{x\in\Omega\, :\, u_k(x)>t\}}\to\chi_{\{x\in\Omega\, :\, u(x)>t\}}$ almost 
everywhere in $\Omega$. 

By Lemma~\ref{min-surface-levels} we know that $\chi_{\{x\in\Omega\, :\, u_k(x)>t\}}$ is of least gradient in
$\Omega$ for each such $t\in\R$, and so by the De Giorgi inequality \eqref{eq:De Giorgi inequality}, we know that whenever $B(x_0,2r)\Subset\Omega$,
\begin{equation}\label{eq:De Giorgi inequality 2}
  P(\{x\in\Omega: u_k(x)>t\},B(x_0,r))\le C\, \frac{\mu(B(x_0,r))}{r}.
\end{equation}
It follows that whenever $K$ is a compact subset of $\Omega$,
\[
   \sup_k P(\{x\in\Omega\, :\, u_k(x)>t\},K)\le \widetilde{C}_K<\infty,
\]
and so by the compact embedding of $BV(K)$ (see~\cite[Theorem 3.7]{M}), there is a subsequence of $\{u_k\}$ such that
$\chi_{\{x\in\Omega\, :\, u_k(x)>t\}}$ converges in $L^1(K)$ to $\chi_{\{x\in\Omega\, :\, u(x)>t\}}$. A diagonalization
argument now yields a subsequence, also denoted by $\{u_k\}$, such that
\[
  \chi_{\{x\in\Omega\, :\, u_k(x)>t\}}\to\chi_{\{x\in\Omega\, :\, u(x)>t\}}
\]
in $L^1_{\loc}(\Omega)$, and hence by Proposition~\ref{stab-min} we know that $\{x\in\Omega\, :\, u(x)>t\}$ is 
of least gradient in $\Omega$. This concludes the proof of the first part of the proposition.

Now suppose that in addition,
\[
  \sup_{k\in\N,\, x\in K}|u_k(x)|=:M_K<\infty
\]
for any compact $K\subset\Omega$.
Then by the coarea formula and \eqref{eq:De Giorgi inequality 2},
\begin{align*}
  \Vert Du_k\Vert(B(x_0,r))&=\int_{-M_{B(x_0,r)}}^{M_{B(x_0,r)}} P(\{x\in\Omega\, :\, u_k(x)>t\},B(x_0,r))\, dt\\
     &\le 2CM_{B(x_0,r)}\, \frac{\mu(B(x_0,r))}{r}<\infty.
\end{align*}
This implies that $C_K:=\sup_k\Vert Du_k\Vert(K)$ is finite for any compact $K\subset\Omega$, so this case reduces to the last case presented in the proposition.

Let us thus assume that whenever $K\subset\Omega$ is compact, we have $C_K<\infty$. Then 
for any ball $B(x_0,\lambda r)\Subset\Omega$ we have by the $(1,1)$-Poincar\'e inequality that
\[
  \sup_k\int_{B(x_0,r)}|u_k-(u_k)_{B(x_0,r)}|\, d\mu \le C\, r\, C_{\overline{B}(x_0,\lambda r)}<\infty.
\]
By using the compactness result \cite[Theorem 3.7]{M} again, the sequence of $BV$ functions $\{u_k-(u_{k})_{B(x_0,r)}\}_k$ has a subsequence that converges in $L^1_{\loc}(B(x_0,r))$ to some function $v\in BV(B(x_0,r))$. By picking a further subsequence if necessary, we have $u_k-(u_k)_{B(x_0,r)}\to v$ pointwise a.e. in $B(x_0,r)$. Since also $u_k\to u$ pointwise and $u$ is finite almost everywhere, the corresponding subsequence of the sequence $\{u_{k,B(x_0,r)}\}_k$ must also converge to some number
$\alpha_{B(x_0,r)}\in\R$. Thus we have $u_k\to u$ in $L^1_{\loc}(B(x_0,r))$, and by the lower semicontinuity of the total variation, $u\in\BV(B(x_0,r))$. Hence $u\in BV_{\loc}(\Omega)$. By covering $\Omega$ with balls and using a diagonal argument, we can pick a subsequence $u_k$ for which $u_k\to u$ in $L^1_{\loc}(\Omega)$, and then it follows from Proposition \ref{stab-min} that $u$ is of least gradient in $\Omega$.
\end{proof}

\section{Continuity of functions of least gradient}

An example in~\cite{HKL} shows that even when the boundary data $f$ is Lipschitz, in general
it is not true that there is a continuous 
solution to the Dirichlet problem for the area functional with
boundary data $f$. A minor modification of the example shows that the same phenomenon occurs also in the case of the Dirichlet problem for least gradients. This is in contrast to the Euclidean situation, where it is known that if the boundary of the domain
has strictly positive mean curvature (in a weak sense) and the boundary data is Lipschitz, then there is exactly one 
Lipschitz solution; see for example~\cite{SWZ}, \cite{P1}, \cite{P2}, \cite{Z}. The example in~\cite{HKL} is
in a Euclidean convex Lipschitz domain, equipped with a $1$-admissible weight in the sense of~\cite{HKM}. Hence even
with the mildest modification of the Euclidean setting, things can go wrong. 

We will show here that in a rather general setting of a metric measure space,
functions of least gradient are continuous everywhere outside their jump sets
(after modification on a set of measure zero  of course).


Recall that a function $u\in L^1_{\loc}(X)$ is approximately continuous outside a set of measure zero.
We can also redefine a function $u\in L^1_{\loc}(X)$ in a set of $\mu$-measure zero so that, outside the 
jump set $S_u:=\{x\in X\, :\, u^\vee(x)>u^\wedge(x)\}$, 
 $u$ is everywhere approximately continuous. To complete the proof that a function of least gradient is continuous outside its jump set, we need the upcoming theorem.

Note that as the characteristic
function of a cardioid shows, approximate continuity need not imply continuity, even under modification on a set of
measure zero. Furthermore, we know that for $u\in BV_{\loc}(X)$, the jump set $S_u$ is of $\sigma$-finite $\mathcal H$-measure; this follows from \cite[Theorem 5.3]{AMP}. Hence our claim that a function of least gradient, after modification on a null set, is continuous everywhere
outside its jump set, is quite strong.

\begin{theorem}\label{cont-min}
Let $\Omega$ be an open set and let $u$ be a function of least gradient in $\Omega$. 
Define $u$ at every point $x\in\Omega$ by choosing the representative $u(x)=u^{\vee}(x)$. If 
$x\in \Omega\setminus S_u$, then $u$ is continuous at $x$.
\end{theorem}

\begin{proof}
Let 
$x\in\Omega\setminus S_u$, so that $u$ is approximately continuous at $x$, and let
\[
t:=u(x)=\aplim\limits_{\Omega\ni y\to x}u(y).
\]
We know that a function of least gradient is locally bounded, see \cite[Theorem~4.2, Remark 3.4]{HKL}. Thus 
$|t|<\infty$. Let $\eps>0$. We now show that, given the choice of representative $u=u^\vee$,
there exists $r_{\eps}>0$ such that $u\ge t-\eps$ in $B(x,r_{\eps})$. By the choice 
of representative $u=u^{\vee}$, it is enough to show that $\mu(B(x,r_{\eps})\setminus E_{t-\eps})=0$, where
\[
E_t:=\left\{x\in\Omega:\,u(x)>t\right\}.
\]
To do so, we will apply the porosity result of~\cite{KKLS} to the level sets of $u$. 

We know from Lemma~\ref{min-surface-levels} that the characteristic function $\chi_{E_{t-\eps}}$ of the super-level set $E_{t-\eps}$ is a function of least gradient, and so $E_{t-\eps}$ is a set of minimal surface (set of minimal boundary
surface in the language of~\cite{KKLS}). Define $\widehat{E}_{t-\eps}$ according to
\begin{equation}\label{eq:definition of porous representative}
\begin{split}
  \widehat{A}:=A\cup & \{x\in X\setminus A\, :\, \exists r_x>0\text{ with }\mu(B(x,r_x)\setminus A)=0\}\\
     &\ \ \ \ \ \ \ \ \ \ \ \ \ \ \setminus \{x\in A\, :\, \exists r_x>0\text{ with }\mu(B(x,r_x)\cap A)=0\}.
\end{split}
\end{equation}
Note that 
$\mu(E_{t-\eps}\Delta \widehat{E}_{t-\eps})=0$; the fact 
that $E_{t-\eps}$ (or $\widehat{E}_{t-\eps}$) is of minimal surface in $\Omega$ now implies that 
$\Omega\setminus \widehat{E}_{t-\eps}$ is locally porous in 
$\Omega$, see \cite[Theorem~5.2]{KKLS}. Now it is enough to show that $x\in\text{int}\, \widehat{E}_{t-\eps}$. 
Because $x\notin S_u$, we have that $u^\vee(x)=u^\wedge(x)=t$, and so
$x\not\in\text{ext}\,\widehat{E}_{t-\eps}$. Thus it suffices to show that
$x\not\in\partial \widehat{E}_{t-\eps}$. Let us assume, contrary to this, that $x\in\partial \widehat{E}_{t-\eps}$. By the 
porosity of $\Omega\setminus \widehat{E}_{t-\eps}$, this means that there exists $r_x>0$ and $C\geq 1$ such that whenever 
$0<r<r_x$, there is a point $z\in B(x,r/2)$ such that
\[
B(z,r/2C)\subset\Omega\setminus \widehat{E}_{t-\eps},
\]
where the constant $C$ is independent of $x$ and $r$. Now $B(z,r/2C)\subset B(x,r)$, and the doubling property 
of the measure gives that
\[
\mu(B(z,r/2C))\geq\gamma \mu(B(x,r)),
\]
where $0<\gamma<1$ is independent of $x$ and $r$. Thus
\[
\limsup\limits_{r\to 0}\frac{\mu\left(\{u> t-\eps\}\cap B(x,r)\right)}{\mu\left(B(x,r)\right)}\leq 1-\gamma<1.
\]
This contradicts the fact that the approximate limit of $u$ at $x$ is $t$, since $u^{\wedge}(x)=t$ implies that
\[
\lim\limits_{r\to 0}\frac{\mu\left(\{u\geq t-\eps/2\}\cap B(x,r)\right)}{\mu\left(B(x,r)\right)}=1.
\]
Therefore $x\not\in\partial \widehat{E}_{t-\eps}$, and hence $x\in\text{int}\, \widehat{E}_{t-\eps}$. As 
noted earlier, this implies that $u\ge t-\eps$ everywhere in $B(x,r_{\eps})$ for some $r_{\eps}>0$. By applying 
a similar argument to the 
sublevel sets $F_t:=\{x\in\Omega: u\le t\}$,
we get
$x\le t+\eps$ in $B(x,r_{\eps})$ for a possibly smaller $r_{\eps}>0$. Thus for every $\eps>0$ there exists 
$r_{\eps}>0$ such that $|u(x)-u(y)|\le\eps$ for all 
$y\in B(x,r)$. Hence $u$ is continuous at $x$.
\end{proof}

\section{Maximum principle}

In this section we prove a 
maximum principle for solutions to the Dirichlet problem for least gradients.
Note that  by considering truncations of the 
approximating sequences in the definition of the total variation, we have the following: if $f\in BV(X)$ with $M_1\le f\le M_2$, 
then there is a solution $u\in BV(X)$ to the Dirichlet problem for least gradients with boundary data $f$
such that  $M_1\le u\le M_2$.
Since we do not have a uniqueness result, this does not automatically imply that all solutions enjoy the same property;
one has to prove the maximum principle independently.

\begin{theorem}\label{max-princ}
If $M\in\R$, and $f\in BV(X)$ is such that $f\le M$ a.e. on $X$,  $\Omega\subset X$ is a
domain such that $\mu(X\setminus\Omega)>0$, and
$u\in BV(X)$ is a solution to the Dirichlet problem for least gradients in $\Omega$ with boundary
data $f$, then $u\le M$ a.e. in $\Omega$.
\end{theorem}

\begin{proof}
Let $u_1:=\min\{u,M\}$ and $u_2:=(u-M)_+$, and note that because $u-f\in BV_0(\Omega)$ and $f\le M$, we also have that $u_1-f\in BV_0(\Omega)$. Hence
by Lemma~\ref{lem:decomposition of BV function by truncation} with $t=M$, 
\[
  \Vert Du_1\Vert(\overline{\Omega})+\Vert Du_2\Vert(\overline{\Omega})
   =\Vert Du\Vert (\overline{\Omega})\le \Vert Du_1\Vert(\overline{\Omega}),
\]
from which we see that $\Vert Du_2\Vert (\overline{\Omega})=0$. 
Therefore by the 
$(1,1)$-Poincar\'e inequality for BV functions, it follows that $u_2$ is locally a.e. constant in $\Omega$, and because $\Omega$ is connected, $u_{2}$ is a.e. constant in $\Omega$. We denote this constant by $L$. 

As pointed out above, $u_1-f\in BV_0(\Omega)$; it follows that since $u_1+u_2-f=u-f\in BV_0(\Omega)$, we must have
$u_2\in BV_0(\Omega)$, that is, $u_2=L\chi_\Omega\in BV(X)$. Therefore if $L\ne 0$, we must have that $P(\Omega,X)$
is finite, and because $\Vert Du_2\Vert(\overline{\Omega})=0$, we must in fact have that
$P(\Omega,X)=0$. 

On the other hand, since $\mu(X\setminus \Omega)>0$, we can find a ball $B=B(x,r)$
with $\mu(B\cap\Omega)>0$ and $\mu(B\setminus\Omega)>0$. Now by the $(1,1)$-Poincar\'e inequality for BV functions, we arrive at the contradiction
\[
0<\min\{\mu(B\cap\Omega),\mu(B\setminus\Omega)\}\le 2\int_B|\chi_{\Omega}-(\chi_{\Omega})_B|\,d\mu\le C r P(\Omega,B(x,\lambda r))=0.
\]
Hence it must be that $L=0$, that is, $u\le M$ a.e. in $\Omega$.
\end{proof}

Unlike in the nonlinear potential theory associated with $p$-harmonic functions for $p>1$, here
we cannot replace the condition $\mu(X\setminus\Omega)>0$ with the requirement $\text{Cap}_1(X\setminus\Omega)>0$, since
there are closed sets of measure zero but with positive $1$-capacity, and if $\mu(X\setminus\Omega)=0$, all functions $u\in\BV(X)$ satisfy any given boundary values $f\in\BV(X)$.

\begin{corollary}
If $M_1,M_2\in\R$, $f\in BV(X)$ is such that $M_1\le f\le M_2$ a.e. on $X$, and $\Omega\subset X$ is a domain such that $\mu(X\setminus\Omega)>0$, then
 every solution to the Dirichlet problem for least gradients in $\Omega$ with boundary data $f$ has the property that 
 $M_1\le u\le M_2$ a.e. in $\Omega$.
\end{corollary}



\section{Regularity of minimizers of the functional $f(t)=\sqrt{1+t^2}$}

The area functional given by the function $f:[0,\infty)\to[0,\infty)$ with $f(t)=\sqrt{1+t^2}$
is much studied in the setting of the Bernstein problem. However, this functional does not satisfy 
either the condition $f(0)=0$ or $\lim_{t\to0^+}(f(t)-f(0))/t>0$, and hence the regularity results obtained so far
in this paper do not apply to this functional. However, by directly considering the meaning of minimizing this
functional, we obtain an analogous regularity result in this section.

While we stick to the model case $f(t)=\sqrt{1+t^2}$, it is easy to verify that the computations and results presented in this section also apply to more general functionals, where $f$ satisfies the growth conditions
\[
m(1+t) \le f(t) \le M(1+t)
\]
for all $t\ge 0$ and some constants $0<m\le M <\infty$. The only difference is that various constants will also depend on $m$ and $M$.

Given an open set $\Omega\subset X$ and a function $u\in L^1_{\loc}(\Omega)$, we define the functional by
\[
  \mathcal{F}(u,\Omega):=\inf\bigg\lbrace
     \liminf_{i\to\infty}\int_\Omega f(g_{u_i})\, d\mu\, :\, N^{1,1}_{\loc}(\Omega)\ni u_i\to u\text{ in }L^1_{\loc}(\Omega)\bigg\rbrace.
\] 
Here each $g_{u_i}$ is an upper gradient of $u_i\in N^{1,1}_{\loc}(\Omega)$. The above definition of $\mathcal{F}$ agrees with that of~\cite[Definition~3.2]{HKL}, since
under the assumption of a $(1,1)$-Poincar\'e inequality and the doubling condition of the measure, we know that Lipschitz continuous functions are dense in $N^{1,1}(X)$ (see for example~\cite{S}), 
and hence locally Lipschitz continuous
functions form a dense subclass of $N^{1,1}_{\loc}(\Omega)$ (see \cite[Theorem 5.47]{BB}).  
Note that if $\mathcal{F}(u,\Omega)$ is finite,
then necessarily $u\in BV_{\loc}(\Omega)$ with $\Vert Du\Vert(\Omega)<\infty$. It is shown
in~\cite{HKLL} that 
$\mathcal{F}(u,\cdot)$ extends to a Radon measure on $\Omega$.

We say that a function $u\in BV_{\loc}(\Omega)$ is a minimizer of the functional
$\mathcal{F}$, or an $\mathcal F$-minimizer, if
whenever $v\in BV_c(\Omega)$, we have $\mathcal{F}(u,K)\le \mathcal{F}(u+v,K)$,
where $K=\text{supt}(v)$.

Minimizers as considered in~\cite{HKL} are global minimizers, where the test functions
$v$ are required only to be in $BV_0(\Omega)$. Our notion of minimizers is a local one in this sense.
Clearly the results in~\cite{HKL} have local analogs in our setting. In particular, we know
from~\cite[Theorem~4.2]{HKL} that minimizers in our sense are locally bounded in $\Omega$.

To study the regularity properties of $\mathcal F$-minimizers, we consider a related metric measure space, $X\times\R$,
equipped with the measure $\mu\times\mathcal{H}^1$ (where $\mathcal{H}^1$ is the Lebesgue measure 
on $\R$), and the metric $d_\infty$ given by
\[
    d_\infty((x,t),(y,s)):=\max\{d(x,y), |t-s|\}
\]
for $x,y\in X$, $t,s\in\R$. Since both $X$ and $\R$ support a $(1,1)$-Poincar\'e inequality and $\mu$, $\mathcal{H}^1$ are
doubling measures, the product space $X\times\R$ equipped with $d_\infty$ and $\mu\times\mathcal{H}^1$
also supports a $(1,1)$-Poincar\'e inequality and $\mu\times\mathcal{H}^1$ is doubling; see~\cite[Proposition~1.4]{Sou}.

Given an $\mathcal{F}$-minimizer $u$ on $\Omega$, we define the subgraph
\[
  E_u:=\{(x,t)\in X\times\R\, :\, x\in\Omega,\, t\le u(x)\}.
\]
Note that for any $\mu$-measurable function $u$, the set $E_u$ is $\mu\times\mathcal H$-measurable in $X\times\R$, see \cite[p. 66]{EvaG92}.

We will use the next theorem in the study of regularity of $\mathcal{F}$-minimizers outside their jump sets.

\begin{theorem}\label{thm:Fminimizer quasiminimal surface}
Let $\Omega\subset X$ be an open set, and let $u\in\BV_{\loc}(\Omega)$ be an $\mathcal{F}$-minimizer. Then $E_u$ is
a set of quasiminimal surface in $\Omega\times\R$ equipped with $d_{\infty}$ and $\mu\times\mathcal{H}^1$.
\end{theorem}

\begin{proof}
Take any open set $U\Subset\Omega$, and let $u_i\in N^{1,1}_{\loc}(U)$ be a sequence with $u_i\to u$ in $L^1_{\loc}(U)$ and upper gradients $g_i$ such that
\[
\lim_{i\to\infty}\int_{U}\sqrt{1+g_i^2}\,d\mu=\mathcal F(u,U).
\]
Now we have for the subgraph $E_u$, by using \cite[Proposition 4.2]{AMP} in the first inequality below,
\begin{equation}
\begin{split}\label{eq:perimeter estimated by functional}
\Vert D\chi_{E_u}\Vert (U\times\R)\le \mu(U)+&\Vert Du\Vert(U) \le \liminf_{i\to\infty}\int_{U}1+g_i\,d\mu\\
&\le2\liminf_{i\to\infty}\int_{U}\sqrt{1+g_i^2}\,d\mu=2\mathcal F(u,U).
\end{split}
\end{equation}

In proving a converse inequality, we need to consider competing sets for $E_u$ that are not necessarily subgraphs of functions. By \cite[Theorem 4.2]{HKL} we know that $u$ is locally bounded, since it is a minimizer of $\mathcal F$. Suppose that $F\subset\Omega\times\R$ is a $\mu\times\mathcal H$-measurable set
such that $\overline{F\Delta E_u}$ is a compact subset of $\Omega\times\R$.
We need to show that
\begin{equation}\label{eq:goal:subgraph is quasiminimizer} 
 \Vert D\chi_{E_u}\Vert(\overline{F\Delta E_u})\le C \Vert D\chi_F\Vert(\overline{F\Delta E_u}).
\end{equation}
Note that by the compactness of $\overline{F\Delta E_u}$, there is an open set $U\Subset\Omega$ that contains the projection of $\overline{F\Delta E_u}$ to $X$.
Since $\overline{F\Delta E_u}$ is compact in $\Omega\times\R$ and $u$ is bounded in $U$, we can assume that $u\ge 0$ in $U$, and that
$(x,t)\in F$ for all $x\in U$ and $t\le 0$.
Let us set $v:U\to[0,\infty]$ to be the function
\[
   v(x)=\int_{0}^\infty \chi_F(x,t)\, d\mathcal{H}^1(t)=\mathcal{H}^1(\{x\}\times[0,\infty) \cap F);
\]
note that this is $\mu$-measurable by Fubini's theorem.

Now $v$ and $u$ may differ only in a compact subset of $U$ (namely, the projection of
$\overline{F\Delta E_u}$ to $X$).
Let $B_i^{\eps}$ be a covering of $U\times\R$ by balls of radius $\eps$, with bounded overlap, where $\eps$ will be chosen shortly,
and let $\pip_i^{\eps}$ be a corresponding partition of unity by Lipschitz functions.
For each $\eps>0$, the discrete convolution
 $\psi_\eps:=\sum_i(\chi_{F})_{B_i^\eps}\pip_i^\eps$ is $C/\eps$-Lipschitz, converges
 to $\chi_{F}$ in $L^1(U\times\R)$ as $\eps\to 0$, and satisfies (below, $U_{C\eps}$ denotes the $C\eps$-neighborhood of $U$)
\begin{equation}\label{eq:lipschitz constant of discrete convolution of F}
 \begin{split}
   \limsup_{\eps\to 0^+}\int_{U\times\R}\Lip\psi_\eps\, d(\mu\times\mathcal{H}^1)
     &\le C\limsup_{\eps\to 0^+}\Vert D\chi_{F}\Vert(U_{C\eps}\times\R)\\
     &\le C\Vert D\chi_{F}\Vert(\overline{U}\times\R) .
 \end{split}
\end{equation}
For the construction of a discrete convolution and its properties, see e.g. the proof of \cite[Theorem 6.5]{KKST1}. Above we need to have $U_{C\eps}\subset\Omega$, but this is true for small enough $\eps$, since $U\Subset\Omega$. We also need to have $\Vert D\chi_{F}\Vert(U_{C\eps}\times\R)<\infty$, but we can assume this, since otherwise \eqref{eq:goal:subgraph is quasiminimizer} necessarily holds. Similarly we can assume that $\Vert D\chi_F\Vert(\partial U\times \R)=0$.

Recall the definition of the local Lipschitz constant $\Lip$ from \eqref{eq:definition of local lipschitz constant}. For a function $h(x,t)$ in the product space $X\times\R$, we will also use the notation
\[
\Lip_X h(x,t):=\limsup_{r\to 0^+}\sup_{y\in B(x,r)\setminus\{x\}}\frac{h(y,t)-h(x,t)}{d(y,x)}.
\] 
Fix $\eps>0$ and set $v_\eps$ on $U$ to be the function
 \[
   v_\eps(x):=\int_{-2\eps}^\infty\psi_\eps(x,t)\, d\mathcal{H}^1(t);
 \]
this clearly converges to $v$ in $L^1(U)$ as $\eps\to 0$.
For $x,y\in U$ we have 
 \[
   |v_\eps(x)-v_\eps(y)|\le \int_{-2\eps}^\infty |\psi_\eps(x,t)-\psi_\eps(y,t)|\, d\mathcal{H}^1(t).
 \]
Here we only need to integrate over a finite interval, since $u$ is bounded in $U$ and $F$ is a perturbation of the subgraph $E_u$ in a bounded set. Since $\psi_{\eps}$ is Lipschitz continuous,
$v_\eps$ is also Lipschitz continuous in $U$. 
By taking a limit as $y\to x$, we get by Lebesgue's dominated convergence theorem
 \[
   \Lip v_\eps(x)\le \int_{-2\eps}^\infty \Lip_X\psi_\eps(x, t)\, d\mathcal{H}^1(t).
 \]
Now note that
 \begin{align*}
  \int_U\Lip v_\eps(x)\, d\mu(x)
  &+\int_U\int_{-2\eps}^\infty |\partial_t\psi_\eps(x,\cdot)(t)|\, d\mathcal{H}^1(t)\, d\mu(x)\\
    &\le \int_U\int_{-2\eps}^\infty [\Lip_X\psi_\eps(\cdot, t)(x)
+|\partial_t\psi_\eps(x,\cdot)(t)|]\, d\mathcal{H}^1(t)\,d\mu(x)\\
    &\le 2\int_{U\times\R}\Lip \psi_\eps(x,t)\, d(\mu\times\mathcal{H}^1)(x,t).
 \end{align*}
By combining this with \eqref{eq:lipschitz constant of discrete convolution of F}, we get
\begin{align*}
 \limsup_{\eps\to 0^+}\Bigg[\int_U\Lip v_\eps(x)\, d\mu(x)
    &+\int_U\int_{-2\eps}^\infty |\partial_t\psi_\eps(x,\cdot)(t)|\, d\mathcal{H}^1(t) \, d\mu(x)\Bigg]\\
  & \le C \Vert D\chi_{F}\Vert(\overline{U}\times\R).
\end{align*}
 On the other hand, by the fact that $\psi_\eps$ is Lipschitz continuous and for each $x\in U$,
 $\psi_\eps(x,t)$ goes from $1$ to $0$ as $t$ increases from $-2\eps$ to $\infty$, we see that
 \[
   \int_{-2\eps}^\infty |\partial_t\psi_\eps(x,\cdot)(t)|\, d\mathcal{H}^1(t)\ge 1.
 \]
 Thus
 \[
 \limsup_{\eps\to 0^+}\int_U\Lip v_\eps(x)\, d\mu(x)
    +\mu(U)\le C \Vert D\chi_{F}\Vert(\overline{U}\times\R).
 \]
It follows that $v\in BV(U)$ with
\begin{equation}\label{eq:functional estimated by perimeter}
\begin{split}
  \mathcal{F}(v,U)\le \limsup_{\eps\to 0}\int_U\sqrt{1+(\Lip v_{\eps})^2}\,d\mu &\le \limsup_{\eps\to 0} \int_U1+\Lip v_{\eps}\,d\mu\\
&\le C \Vert D\chi_F\Vert(\overline{U}\times\R).
\end{split}
\end{equation}
Now, by using \eqref{eq:perimeter estimated by functional}, the fact that $u$ is an $\mathcal{F}$-minimizer, and finally \eqref{eq:functional estimated by perimeter}, we see that
\begin{align*}
 \Vert D\chi_{E_u}\Vert (U\times\R)\le 2\mathcal{F}(u,U)
     &\le 2\mathcal{F}(v,U)\\
     &\le C  \Vert D\chi_F\Vert(\overline{U}\times\R)=C\Vert D\chi_F\Vert(U\times\R),
\end{align*}
since we had $\Vert D\chi_F\Vert(\partial U\times \R)=0$.
The above implies \eqref{eq:goal:subgraph is quasiminimizer}, so that $E_u$ is a set of quasiminimal 
surface in $\Omega\times\R$.
\end{proof}

A more thorough analysis of the relationship between the area functional $\mathcal F(u,\Omega)$ and the 
perimeter of the subgraph $P(E_u,\Omega\times\R)$ has recently been conducted by Ambrosio et al. in \cite{APS}
under an additional assumption on the metric space $X$ which is different from the assumptions made in this 
paper. Indeed, \cite{APS} considers functions on more general product spaces $X\times Y$ with $X,Y$ satisfying
an assumption concerning equality between two notions of minimal gradients, whereas we consider the simpler case 
$Y=\mathbb{R}$. Given these differing assumptions, neither the following theorem nor~\cite[Theorem~5.1]{APS} is a 
special case of the other.

As in Section~4, by modifying a function $u\in BV_{\loc}(\Omega)$ on a 
set of measure zero if necessary, we can assume that whenever $x\in\Omega\setminus S_u$, we have
\[
  \aplim_{y\to x}u(y)=u(x).
\]
That is, every point in $\Omega\setminus S_u$ is a point of approximate continuity of $u$.

\begin{theorem} 
Let $\Omega\subset X$ be an open set, and let $u\in BV_{\loc}(\Omega)$ be an $\mathcal F$-minimizer. 
Define $u$ at every point $x$ by choosing the representative $u(x)=u^{\vee}(x)$. If $x\in\Omega\setminus S_u$,
then $u$ is continuous at $x$.
\end{theorem}

\begin{proof}

According to Theorem \ref{thm:Fminimizer quasiminimal surface}, $E_u$ is a set of quasiminimal surface in $\Omega\times\R\subset X\times\R$ equipped with 
$d_\infty$ and $\mu\times\mathcal{H}^1$, and hence by \cite{KKLS} has the properties of 
$[\Omega\times\R]\cap \partial^*\widehat{E}_u=[\Omega\times\R]\cap\partial \widehat{E}_u$
and porosity,
where $\widehat{E}_u$ is defined according to \eqref{eq:definition of porous representative}.
Since $[\Omega\times\R]\cap\partial^*\widehat{E}_u=[\Omega\times\R]\cap\partial \widehat{E}_u$, for any $x\in\Omega$ we necessarily have $(x,t)\in \widehat{E}_u$ for $t<u^{\wedge}(x)$ and $(x,t)\notin \widehat{E}_u$ for $t>u^{\vee}(x)$.
Moreover, whenever $u^{\wedge}(x)\le t\le u^{\vee}(x)$ and $r>0$, we must have
\begin{equation}\label{eq:subgraph jump set}
\mu\times\mathcal H^1(B((x,t),r)\cap \widehat{E}_u)>0,\qquad\mu\times\mathcal H^1(B((x,t),r)\setminus \widehat{E}_u)>0.
\end{equation}
If we add all points $(x,t)$ with $t\le u^{\vee}(x)$ to $\widehat{E}_u$, we get precisely the subgraph $E_u$ of $u=u^{\vee}$. Clearly $\mu\times\mathcal H^1(E_u\Delta\widehat{E}_u)=0$, and by \eqref{eq:subgraph jump set} we have $\partial E_u=\partial \widehat{E}_u$, so that we also have $[\Omega\times\R]\cap \partial^*E_u=[\Omega\times\R]\cap\partial E_u$, and the porosity property holds for $E_u$ and its complement.
 
Let $x\in\Omega\setminus S_u$, so that $x$ is a point of approximate continuity of $u$, and let $t:=u(x)$. Since $\mathcal F$-minimizers are locally bounded by \cite[Theorem 4.1]{HKL}, we have $|t|<\infty$. We wish to show that whenever
$\eps>0$, there is a ball $B(x,20r)\subset\Omega$ such that $u>t-\eps$ on $B(x,r)$. Suppose not, then
there is some $\eps>0$ such that whenever $r>0$ with $B(x,20r)\subset \Omega$, there is some 
$y\in B(x,r)$ with $u(y)\le t-\eps$.  In particular, we can choose $0<r<\eps/2$.

Since $(y,u(y))\in \partial E_u$, by the porosity of $X\setminus E_u$ (see \cite[Theorem 5.2]{KKLS}) we know that there is some point $(w,s)$ in the ball
(with respect to the metric $d_\infty$) $B(y,r)\times(u(y)-r,u(y)+r)$ such that the $d_\infty$-ball 
$B(w,r/C)\times(s-r/C,s+r/C)$ lies in $[\Omega\times\R]\setminus E_u$. Note that to use this result, we need $B(x,20r)\subset\Omega$. Now for a.e.
$z\in B(w,r/C)$, we have $u(z)\le s-r/C$. Since $|s-u(y)|<r$, we have
\[
  s<r+u(y)\le r+t-\eps.
\]
Thus for a.e. $z\in B(w,r/C)\subset B(x,3r)$,
\[
  u(z)<(r+t-\eps)-\frac{r}{C}<r+t-\eps<\frac{\eps}{2}+t-\eps=t-\frac{\eps}{2}.
\]
Since this is true for every $r<\eps/2$, and the measure $\mu(B(w,r/C))$ is comparable to $\mu(B(x,3r))$ by the doubling property of $\mu$, we cannot have the approximate limit of $u$ at $x$ be $t$. This is a 
contradiction of the assumption that $x$ is a point of approximate continuity of $u$ with $u(x)=t$. Thus for each 
$\eps>0$, there is some $r>0$ such that $u>t-\eps$ on $B(x,r)$.

A similar argument shows that for each $\eps>0$, there is some $r>0$ such that $u<t+\eps$ on $B(x,r)$. That is,
$u$ is continuous at $x$.
\end{proof}

Recall that even in a weighted Euclidean setting we cannot insist on $u$ being continuous everywhere (that is, we cannot insist on the 
jump set $S_u$ being empty), as demonstrated by~\cite[Example~5.2]{HKL}.

\noindent Address:\\

\noindent (H.H.): Department of Mathematical Sciences, P.O. Box 3000, FI-90014 University of Oulu, Finland. \\
E-mail: {\tt heikki.hakkarainen@oulu.fi}

\medskip

\noindent (R.K.): Department of Mathematics and Statistics, P.O. Box 68, FI-00014 University of Helsinki, Finland. \\
E-mail: {\tt riikka.korte@helsinki.fi} 

\medskip

\noindent (P.L.): Aalto University, School of Science and Technology, Department of 
Mathematics, P.O. Box 11100, FI-00076 Aalto, Finland. \\ 
E-mail: {\tt panu.lahti@aalto.fi}

\medskip

%

\noindent (N.S.): Department of Mathematical Sciences, P.O.Box 210025, University of Cincinnati, Cincinnati, OH 45221{0025, U.S.A. \\
E-mail: {\tt shanmun@uc.edu}

\end{document}